\pdfoutput=1
\documentclass[12pt,reqno]{amsart}
\usepackage[utf8]{inputenc}
\usepackage[letterpaper,margin=1in,footskip=0.25in]{geometry}
\usepackage{mathrsfs}
\usepackage{amssymb}
\usepackage{mathtools}
\usepackage{setspace}
\usepackage{tikz-cd}
\usepackage{biblatex}
\usepackage{enumitem}

\PassOptionsToPackage{pdfusetitle,colorlinks}{hyperref}
\usepackage{bookmark}
\hypersetup{
  linkcolor={blue!70!black},
  citecolor={green!70!black},
  urlcolor={teal!90!black}
  }

\newtheorem{theorem}{Theorem}[section]
\newtheorem{lemma}[theorem]{Lemma}
\newtheorem{proposition}[theorem]{Proposition}
\newtheorem{corollary}[theorem]{Corollary}

\newtheorem{step}{Step}[subsection]

\theoremstyle{definition}

\theoremstyle{remark}

\newcommand{\sA}{\mathcal{A}}
\newcommand{\sB}{\mathcal{B}}

\newcommand{\sD}{\mathcal{D}}

\newcommand{\sH}{\mathcal{H}}

\newcommand{\sK}{\mathcal{K}}
\newcommand{\sL}{\mathcal{L}}

\newcommand{\sO}{\mathcal{O}}
\newcommand{\sP}{\mathcal{P}}

\newcommand{\sR}{\mathcal{R}}

\newcommand{\sV}{\mathcal{V}}

\newcommand{\sX}{\mathcal{X}}
\newcommand{\sY}{\mathcal{Y}}

\newcommand{\bbC}{\mathbb{C}}
\newcommand{\bbG}{\mathbb{G}}

\newcommand{\bbP}{\mathbb{P}}
\newcommand{\bbQ}{\mathbb{Q}}

\newcommand{\bbZ}{\mathbb{Z}}

\newcommand{\onto}{\rightarrow\hspace*{-.14in}\rightarrow}

\DeclareMathOperator{\Alb}{Alb}

\DeclareMathOperator{\Div}{Div}
\DeclareMathOperator{\Pic}{Pic}
\DeclareMathOperator{\coker}{coker}
\DeclareMathOperator{\im}{im}

\DeclareMathOperator{\id}{id}
\DeclareMathOperator{\Hom}{Hom}

\DeclareMathOperator{\Jac}{Jac}
\DeclareMathOperator{\Aut}{Aut}

\DeclareMathOperator{\Prym}{Prym}

\DeclareMathOperator{\Gal}{Gal}

\DeclareMathOperator{\End}{End}

\DeclareMathOperator{\pr}{pr}

\DeclareMathOperator{\Corr}{Corr}

\begin{document}
\title[On The Generic Prym of a Cyclic Covering] {The generic isogeny decomposition of the Prym Variety of a cyclic branched covering}

\author{Theodosis Alexandrou}
\setstretch{1.35}

\subjclass[2010]{}

\makeatletter
  \hypersetup{
    pdfauthor=,
    pdfsubject=\@subjclass,
    pdfkeywords=\@keywords
  }
\makeatother
\maketitle
\begin{abstract} Let $f\colon S'\longrightarrow S$ be a cyclic branched covering of smooth projective surfaces over $\bbC$ whose branch locus $\Delta\subset S$ is a smooth ample divisor. Pick a very ample complete linear system $|\sH|$ on $S$, such that the polarized surface $(S, |\sH|)$ is not a scroll nor has rational hyperplane sections. For the general member $[C]\in|\sH|$ consider the $\mu_{n}$-equivariant isogeny decomposition of the Prym variety $\Prym(C'/C)$ of the induced covering $f\colon C'\coloneqq f^{-1}(C)\longrightarrow C$:\[\Prym(C'/C)\sim\prod_{d|n,\ d\neq1}\sP_{d}(C'/C).\] We show that for the very general member $[C]\in|\sH|$ the isogeny component $\sP_{d}(C'/C)$ is $\mu_{d}$-simple with $\End_{\mu_{d}}(\sP_{d}(C'/C))\cong\bbZ[\zeta_{d}]$. In addition, for the non-ample case we reformulate the result by considering the identity component of the kernel of the map $\sP_{d}(C'/C)\subset\Jac(C')\longrightarrow\Alb(S')$.
\end{abstract}

\begin{section}{Introduction}\label{sec:1}
The main result of this paper is the following:
\begin{theorem}\label{thm:1.1} Let $S$ be a smooth projective surface over $\bbC$ with an ample line bundle $\sL$. Assume $\Delta\in|\sL^{\otimes n}|$ is smooth and consider the $n$-fold cyclic covering $f\colon {S}'\longrightarrow S$ branched along the divisor $\Delta$. Given a very ample complete linear system $|\sH|$ on $S$, such that $(S,|\sH|)$ is not a scroll nor has rational hyperplane sections. Then, for the very general member $[C]\in |\sH|$ we have that \[\Prym({C}'/C)\sim\prod_{d|n,\ d\neq 1}\sP_{d}(C'/C),\]with $\End_{\mu_{d}}(\sP_{d}(C'/C))\cong \mathbb{Z}[\zeta_{d}]$. Especially, each $\sP_{d}(C'/C)$ is a $\mu_{d}$-simple abelian variety.
\end{theorem}
If we restrict to the case of double coverings, we note that the involution $\sigma$ of the covering $f$ acts as $-\id$ on $\sP_{2}(C'/C)=\Prym(C'/C)$ and thus, $\End_{\mu_{2}}(\Prym(C'/C))=\End(\Prym(C'/C))$. In particular, \eqref{thm:1.1} can be stated as follows:
 \begin{corollary}\label{thm:1.2} Let $S$ be a smooth projective surface over $\bbC$ with an ample line bundle $\sL$. Assume $\Delta\in|\sL^{\otimes 2}|$ is smooth and consider the double covering $f\colon {S}'\longrightarrow S$ branched along the divisor $\Delta$. Given a very ample complete linear system $|\sH|$ on $S$, such that $(S,|\sH|)$ is not a scroll nor has rational hyperplane sections. Then, for the very general member $[C]\in |\sH|$ we have that \[\End(\Prym(C'/C))\cong \bbZ.\]\end{corollary}
The condition the line bundle $\sL$ is ample in \eqref{thm:1.1} implies that $\Alb(f)\colon\Alb(S')\longrightarrow\Alb(S)$ is an isomorphism cf. page \pageref{10} and therefore the map $\sP_{d}(C'/C)\longrightarrow\Alb(S')$ is trivial. For the general situation one needs to consider the abelian subvariety \[\sR_{d}(C',C,S')\coloneqq\ker^{0}(\sP_{d}(C'/C)\longrightarrow\Alb(S')).\] Then, the result can be reformulated as follows: 
\begin{theorem}\label{thm:1.3} Let $S$ be a smooth projective surface over $\bbC$ with a line bundle $\sL$. Assume $\Delta\in|\sL^{\otimes n}|$ is smooth and consider the $n$-fold cyclic covering $f\colon {S}'\longrightarrow S$ branched along the divisor $\Delta$. Given a very ample complete linear system $|\sH|$ on $S$, such that $(S,|\sH|)$ is not a scroll nor has rational hyperplane sections. Then, exactly one of the following assertions holds true:
\begin{enumerate}
\item[$(\rm{i})$] For the general member $[C]\in|\sH|$ we have that $\sR_{d}(C',C,{S}')=0$.\\
\item[$(\rm{ii})$] For the very general member $[C]\in|\sH|$ we have that $\End_{\mu_{d}}(\sR_{d}(C',C,S'))\cong\bbZ[\zeta_{d}]$.
\end{enumerate}
\end{theorem}
In this paper we present a complete proof for the above results, inspired by Ciliberto and Van der Geer's approach in \cite{cv}. We note that this method does not capture the étale situation, cf. \eqref{lem:3.2}, \eqref{prop:3.3} and \eqref{lem:3.4}. In addition, if we rephrase the statement for $n > 2$ by requiring simplicity instead of $\mu_{d}$-simplicity to the isogeny components, we observe that this method cannot be adopted. Namely, the abelian variety $B$ in \eqref{lem:3.4} cannot be chosen in general to be $\mu_{d}$-invariant and for this reason the last combinatorial argument in \eqref{lem:3.4} fails. Lastly, a result due to Ortega and Lange, cf. \cite{ol} may be used to find counter-example for the case the covering $f$ is étale of degree $7$. \\ \\
{\bf Acknowledgements.} The author thanks his advisor Professor Dr. Daniel Huybrechts and Dr. Gebhard Martin for helpful discussions on the topics of this note, which are parts of the author's Master thesis, as well as for corrections.
\end{section}
\begin{section}{Preliminaries}\label{sec:2} In this section, we state some well-known results, which are needed later.
\begin{proposition}\label{prop:2.1} Let $\pi\colon\sA\longrightarrow S$ be a projective abelian scheme over a Noetherian base $S$. Then, the endomorphism functor of $\sA$ over $S$ is representable by an $S$-scheme $\End_{\sA/S}$, which is a disjoint union of projective and unramified $S$-schemes.
\end{proposition}
\begin{proof} This is well-known, cf. \cite[pp.\ 133]{nit}.
\end{proof}
The following proposition relates the correspondences on $C\times C$ with the endomorphisms of the Jacobian $\Jac(C)$.
\begin{proposition}\label{prop:2.2} Let $\pi\colon\sX \longrightarrow S$ be a projective smooth morphism over a Noetherian base $S$, whose fibres are geometrically integral curves. Furthermore, assume that the morphism $\pi$ admits a section, i.e. $\sX(S)\neq\emptyset$. Then, there is a natural and functorial isomorphism \[\Corr_{S}(\sX)\coloneqq\Pic(\sX\times_{S}\sX)/({\pr_{1}})^{*}\Pic(\sX)\otimes({\pr_{2}})^{*}\Pic(\sX)\cong\End_{S}(\Pic^{0}_{\sX/S}). \]
\end{proposition}
\begin{proof} Consider the commutative diagram:\begin{center}\begin{tikzcd}
0 \arrow[r] & \Pic(\sX)/\pi^{*}\Pic(S) \arrow[r, "(\pr_{1})^{*}"] \arrow[d, "\cong"'] & \Pic(\sX\times_{S}\sX)/(\pr_{2})^{*}\Pic(\sX) \arrow[r, "q"] \arrow[d, "\cong"] &\Corr_{S}(\sX)  \arrow[d, dashed, "g"] \arrow[r] & 0 \\
0 \arrow[r] & \Pic_{\sX/S}(S) \arrow[r, "c\coloneqq-\circ \pi"]                                 & \Pic_{\sX/S}(\sX) \arrow[r, "d"]                                     & \End_{S}(\Pic^{0}_{\sX/S}) \arrow[r]                      & 0
\end{tikzcd}\end{center} The first row is clearly exact: Indeed, the relative Picard functor is an fppf-sheaf, cf. \cite[Thm.\ 2.5]{kl} and thus, the restriction map $(\pr_{1})^{*}$ is injective. Furthermore, the map $q$ is just the cokernel of $(\pr_{1})^{*}$. Next, we give the definition of the map $d$. Fix $x\in\sX(S)$ and let $\phi\colon\sX\longrightarrow\Pic_{\sX/S}$ be any $S$-morphism. Then, $d\phi$ is the unique endomorphism of $\Pic^{0}_{\sX/S}$, making the diagram below commutative.
\begin{center}\begin{tikzcd}
\sX \arrow[r, "can"] \arrow[d, "\phi-\phi\circ x\circ \pi"'] \arrow[rd, dashed] & \Alb_{\sX/S}\cong \Pic^{0}_{\sX/S} \arrow[d, "d\phi", dashed] \\
\Pic_{\sX/S}                                                                    & \Pic^{0}_{\sX/S} \arrow[l, hook']                            
\end{tikzcd}\end{center} Note that under our assumptions the Albanese map $can\colon\sX\longrightarrow\Alb_{\sX/S}$ exists and has the desired universal property, cf. \cite[Thm.\ 2.17]{ant}, \cite[Rem.\ 2.19]{ant} and \cite[Thm.\ 10.2]{lsc}. Moreover, the construction of the map $d$ indicates that $d$ is surjective and also that the second row in the diagram above is exact at the middle. Now, the existence of $g$ and the fact that it is an isomorphism are clear.
\end{proof}
The following proposition is well-known. 
\begin{proposition}\label{prop:2.3}Suppose that the polarized surface $( S, |\sH| )$ is not a scroll nor has rational hyperplane sections. Then, the following assertions hold true:
\begin{enumerate}
\item[$(\rm{i})$] The discriminant divisor $\sD$ is irreducible and has codimension one in $ | \sH  |$, i.e. $\sD$ is a prime divisor of $ | \sH  |$.\\
\item[$(\rm{ii})$] The general curve $[ C ]\in \sD$ is irreducible and has a single ordinary double point as its only singularity.
\end{enumerate}
\end{proposition}
\begin{proof} Cf. \cite[Lem.\ 3.1]{cv}.
\end{proof}
We close this section by introducing the $\mu_{n}$-equivariant isogeny decomposition in \eqref{thm:1.1}. Let $f\colon C'\longrightarrow C$ be a cyclic branched covering of smooth complex projective curves with $\deg(f)=n$ and let $\sigma$ stand for a generator of the Galois group of $f$. The $\mu_{n}$-action on $C'$ induces an action on $\Jac(C')$ and thus, it defines a $\bbQ$-algebra homomorphism \[\rho\colon\bbQ[\mu_{n}]\cong\bbQ[T]/(T^{n}-1)\longrightarrow\End^{0}(\Jac(C')),\ T\mapsto \sigma^{*} .\] For any divisor $d|n$, we define $\sP_{d}(C'/C)\coloneqq\ker^{0}(\Psi_{d}(\sigma^{*}))$, where $\Psi_{d}(T)\in\bbZ[T]$ is the $d$-th cyclotomic polynomial. Then, the addition map \[\mu\colon\prod_{d|n}\sP_{d}(C'/C)\longrightarrow\Jac(C')\] is a $\mu_{n}$-equivariant isogeny. Lange and Recillas \cite{lr} have stated and proved the relation between $\mathbb{Q}$-representations and the $G$-equivariant isogeny decomposition of an abelian variety with $G$-action, in terms of the finite group $G$ involved, cf. \cite[Thm.\ 2.2]{lr}. The $\mu_{n}$-equivariant isogeny decomposition of $\Jac(C')$ given above is in fact identical with the one introduced by Lange and Recillas \cite{lr}. This can be seen for example by using \cite[Rem.\ 5.5]{cr} and \cite[Cor.\ 5.7]{cr}. Moreover, we also note that the isogeny components $\sP_{d}(C'/C)$ are non-trivial as long as the genus $g(C)\geq1$, cf. \cite[Thm.\ 3.1]{lr}, \cite[Thm.\ 5.12]{roj} and \cite[Thm.\ 5.13]{roj}.
\end{section}
\begin{section}{Reduction to the generic fibre}\label{sec:3}
\par Let $S$ be a smooth projective surface over $\mathbb{C}$ with an ample line bundle $\sL$. Assume $\Delta\in|\sL^{\otimes n}|$ is smooth and consider the $n$-fold cyclic covering $f \colon {S}' \longrightarrow S$ branched along the divisor $\Delta$. Furthermore, fix a very ample complete linear system $| \sH  |$ on $S$, such that the polarized surface $( S,| \sH |  )$ is not a scroll nor has rational hyperplane sections. In this section we reduce the proof of Theorem \ref{thm:1.1} to showing that $\sP_{d}(C'_{\eta}/C_{\eta})$ is a $\mu_{d}$-simple abelian variety, where $[C_{\eta}]$ is the generic member of $|\sH|$.\par Let $x\in S$ be a closed point of $S$. We denote by $|\sH|_{x}$ the linear system of hyperplane sections in $|\sH|$ passing through $x$. In the following we impose restrictions on the point $x$, i.e. $x\in S$ will be taken from some appropriate non-empty open subset of $S$.\par Let $g\colon \sX\subset S\times  | \sH  |_{x}\longrightarrow  | \sH  |_{x}$ denote the universal family of hyperplane sections and $h \colon \sY\subset {S}'\times  |\sH  |_{x}\longrightarrow  | \sH  |_{x}$ its pullback to ${S}'$, i.e. $\sY\coloneqq\sX\times _{S} {S}'$. Note that over the non-empty open subset $U\subset|\sH|_{x}$ of smooth curves which intersect the branch locus $\Delta$ transversally both $g$ and $h$ are smooth families of curves having a section. The latter allows us to consider their families of Jacobians over $U$, which we denote by $p \colon \Pic^{0}_{\sX/U}\longrightarrow U$ and $q \colon \Pic^{0}_{\sY/U}\longrightarrow U$, respectively.\par A generator $\sigma\colon S'\longrightarrow S'$ of the Galois group of the covering $f$ induces an automorphism of $\sY$ over $U$ and thus, an automorphism $\sigma^{*}\colon\Pic^{0}_{\sY/U}\longrightarrow\Pic^{0}_{\sY/U}$. We define \[\sP_{d}\coloneqq\ker^{0}(\Psi_{d}(\sigma^{*}))\text{ for any divisor}\ d|n.\] Then, $\varphi_{d}\colon\sP_{d}\longrightarrow U$ is an abelian fibration with fibres ${(\sP_{d})}_{[C]}=\sP_{d}(C'/C)$ for $[C]\in U$.
\par As a first step we use the representability of the endomorphism functor of abelian schemes cf. \eqref{prop:2.1} to reduce the proof of Theorem \ref{thm:1.1} to showing that $\End_{\mu_{d}}((\sP_{d})_{\bar{\eta}})\cong\bbZ[\zeta_{d}]$, where $\bar{\eta}$ is a fixed geometric generic point of $|\sH|_{x}$. The proof of this is standard and so we omit it. 
\begin{lemma}\label{lem:3.1}  Assume that $\End_{\mu_{d}}((\sP_{d})_{\bar{\eta}})\cong \bbZ[\zeta_{d}]$. Then, for the very general member $[C]\in U$, one has that $\End_{\mu_{d}} ((\sP_{d})_{[C]} )\cong \bbZ[\zeta_{d}]$.
\end{lemma}
Let $[C]\in|\sH|_{x}$ be an irreducible member with a single ordinary double point as its only singularity and intersecting the branch locus $\Delta$ transversally. Then, $C'\coloneqq f^{-1}(C)$ is irreducible and has $n$ ordinary double points as its only singularities. In this case the group variety $\sP_{d}(C'/C)$ is semi-abelian. In particular, the result is the following:
\begin{lemma}\label{lem:3.2} For an irreducible member $[C]\in|\sH|_{x}$ with a single ordinary double point as its only singularity and intersecting the branch locus $\Delta$ transversally, there is an exact sequence:
\begin{center}\begin{tikzcd}
0 \arrow[r] & \mathbb{G}_{m}^{\varphi(d)} \arrow[r, hook] & {\sP_{d}(C'/C)} \arrow[r, two heads] & \sP_{d}(\tilde{C'}/\tilde{C}) \arrow[r] & 0,
\end{tikzcd}
\end{center}
where $\nu\colon\tilde{C}\longrightarrow C$ is the normalisation map and $\varphi(d)$ is the Euler's totient function.
\end{lemma}
\begin{proof}We have a commutative diagram\begin{center}\begin{tikzcd}
\tilde{C'} \arrow[r, "\nu'"] \arrow[d, "\tilde{f}"] & C' \arrow[d, "f"] \\
\tilde{C} \arrow[r, "\nu"]                          & C,                
\end{tikzcd}
\end{center}
where $\tilde{f}$ is the cyclic covering branched along the divisor $\nu^{*}\Delta|_{C}\in|\nu^{*}\sL|_{C}^{\otimes n}|$ and $\nu'$ is the normalisation of $C'$. Fix a generator $\sigma$ of $\Aut(C'/C)$ and let $\tilde{\sigma}$ be the corresponding generator of $\Aut(\tilde{C'}/\tilde{C})$, i.e. the one for which the diagram below commutes \begin{center}\begin{tikzcd}
\tilde{C'} \arrow[r, "\nu'"] \arrow[d, "\tilde{\sigma}"] & C' \arrow[d, "\sigma"] \\
\tilde{C'} \arrow[r, "\nu'"]                            & C'.                  
\end{tikzcd}
\end{center}
Let $\{y,\sigma(y),\sigma^{2}(y),\ldots,\sigma^{n-1}(y)\}$ be the set of ordinary double points of $C'$. Then, we find a commutative diagram with exact rows and columns \begin{center}\begin{tikzcd}
0 \arrow[r] & \ker(\alpha) \arrow[r, hook]                                                                               & \Psi_{d}(\sigma^{*})\Pic^{0}(C') \arrow[r, "\alpha", two heads]      & \Psi_{d}(\tilde{\sigma}^{*})\Pic^{0}(\tilde{C'}) \arrow[r] & 0 \\
0 \arrow[r] & \mathbb{C}^{*}_{y}\times\dots\times\mathbb{C}^{*}_{\sigma^{n-1}(y)} \arrow[r, hook] \arrow[u, "\gamma"]      & \Pic^{0}(C') \arrow[r, "\nu'^{*}", two heads] \arrow[u, two heads] & \Pic^{0}(\tilde{C'}) \arrow[r] \arrow[u, two heads]      & 0 \\
0 \arrow[r] & \mathbb{C}^{*}_{y}\times\dots\times\mathbb{C}^{*}_{\sigma^{\varphi(d)-1}(y)} \arrow[u, hook] \arrow[r, hook] & \ker(\Psi_{d}(\sigma^{*})) \arrow[u, hook] \arrow[r, "\beta"]        & \ker(\Psi_{d}(\tilde{\sigma}^{*})) \arrow[u, hook]  .       &  
\end{tikzcd}\end{center} We show that $\beta$ induces a surjection $\sP_{d}(C'/C)=\ker^{0}(\Psi_{d}(\sigma^{*}))\onto \sP_{d}(\tilde{C'}/\tilde{C})=\ker^{0}(\Psi_{d}(\tilde{\sigma}^{*}))$. Indeed, by Snake lemma we have the exact sequence \[\ker(\Psi_{d}(\sigma^{*}))\longrightarrow\ker(\Psi_{d}(\tilde{\sigma}^{*}))\longrightarrow \coker(\gamma)\longrightarrow 0.\] Note that $\coker(\gamma)$ is an affine algebraic group, as it is the quotient of a commutative affine algebraic group by an algebraic subgroup. Especially, by \cite[Cor.\ 12.67]{gw} and the exactness of the above sequence, we find that $\dim\coker(\gamma)=0$. The latter provides the surjectivity of the map $\ker^{0}(\Psi_{d}(\sigma^{*}))\longrightarrow \sP_{d}(\tilde{C'}/\tilde{C})=\ker^{0}(\Psi_{d}(\tilde{\sigma}^{*}))$, as claimed.
\end{proof}
We are now in the position to prove the following:
\begin{proposition}\label{prop:3.3} The abelian variety $(\sP_{d})_{\bar{\eta}}$ is $\mu_{d}$-simple if and only if $\End_{\mu_{d}}((\sP_{d})_{\bar{\eta}})\cong\bbZ[\zeta_{d}]$.
\end{proposition}
\begin{proof} The one direction is clear: Indeed, if $\End_{\mu_{d}}((\sP_{d})_{\bar{\eta}})\cong\bbZ[\zeta_{d}]$, then every non-zero $\mu_{d}$-equivariant endomorphism of $(\sP_{d})_{\bar{\eta}}$ is an isogeny and thus, $(\sP_{d})_{\bar{\eta}}$ is a $\mu_{d}$-simple abelian variety. Conversely, assume that $(\sP_{d})_{\bar{\eta}}$ is $\mu_{d}$-simple. We divide the proof into steps.
\begin{step}\label{1} There is a closed subscheme $\End^{\mu_{d}}_{\sP_{d}/U}(0)\subset \End^{\mu_{d}}_{\sP_{d}/U}$ whose points parametrise the $\mu_{d}$-equivariant endomorphisms of $\sP_{d}$, which are not isogenies, i.e. the ones, which are of degree $0$.
\begin{proof}[Proof of Step 1] Observe that the functor of $\mu_{d}$-equivariant endomorphisms of $\sP_{d}$ denoted by $\End^{\mu_{d}}_{\sP_{d}/U}$ is representable by a closed subscheme of $\End_{\sP_{d}/U}$, since the equivariant condition is closed. It follows that we have a universal endomorphism $\alpha$, such that every other $\mu_{d}$-equivariant endomorphism of $\sP_{d}$ over some scheme $T$ is obtained by pulling-back $\alpha$ along a morphism $T\longrightarrow \End^{\mu_{d}}_{\sP_{d}/U}$. By \cite[Prop.\ 12.93]{gw} the set \[\sV\coloneqq\{ x\in\End^{\mu_{d}}_{\sP_{d}/U}| \ \alpha_{x}\coloneqq \alpha\times \id_{\kappa ( x  )} \ \text{is an isogeny} \}\] is open. Therefore, $\End^{\mu_{d}}_{\sP_{d}/U}(0)\coloneqq \End^{\mu_{d}}_{\sP_{d}/U}\setminus \sV$ with the reduced induced closed subscheme structure has the desired property.
\end{proof}
\end{step}
\begin{step}\label{2} The fibre $(\sP_{d})_{[C]}$ for the very general member $[C]\in|\sH|_{x}$ is a $\mu_{d}$-absolutely simple abelian variety.
\begin{proof}[Proof of Step 2] Since $(\sP_{d})_{\bar{\eta}}$ is a $\mu_{d}$-simple abelian variety, we can determine countably many non-empty open subsets $U_{i}\subset U$, such that the $U$-scheme $\End^{\mu_{d}}_{\sP_{d}/U}(0)$ has (geometrically) connected fibres for all points lying in the intersection of the $U_{i}$'s, cf. \cite[\href{https://stacks.math.columbia.edu/tag/055C}{Tag 055C}]{sta}.
\end{proof}
\end{step}
Pick a Lefschetz pencil $(C_{t})_{t\in\bbP^{1}}\subset|\sH|_{x}$. We may assume that all its singular members are irreducible and intersect the branch locus $\Delta$ transversally, cf. \eqref{prop:2.3}. 
\begin{step}\label{3} Given a Lefschetz pencil $(C_{t})_{t\in\mathbb{P}^{1}}$ as above, we construct a homomorphism:  \[\rho \colon \End_{\mu_{d}}((\sP_{d})_{\bar{\mu}}) \longrightarrow \End(\mathbb{G}^{\varphi(d)}_{m}),\] where $\bar{\mu}$ is a fixed geometric generic point of $\mathbb{P}^{1}$.
\begin{proof}[Proof of Step 3] Since the endomorphism ring of any abelian variety is finitely generated, cf. \cite[Thm.\ 12.5]{mil1}, we find a finite field extension $L\supset\kappa(\mu)$, such that every endomorphism of $\sP_{d}$ over $\kappa(\bar{\mu})$ is defined over $L$, i.e. $\End((\sP_{d})_{\bar{\mu}})=\End((\sP_{d})_{L})$. Consider the smooth projective model $E$ of $L$ together with the morphism $E\longrightarrow\bbP^{1}$ induced by this field extension and fix a closed point $y\in E$ lying over a point of the pencil that corresponds to a nodal curve. The map $\rho \colon \End_{\mu_{d}}((\sP_{d})_{\bar{\mu}}) \longrightarrow \End(\mathbb{G}^{\varphi(d)}_{m})$ is constructed as follows: Let $f\in\End_{\mu_{d}}((\sP_{d})_{L})$. Then, $f$ extends to an endomorphism over the local ring $R$ of $E$ at the point $y$, cf. \cite[Prop.\ 7.4.3]{blr}. The restriction of the first projection of $\sP_{d} \times _{R}\sP_{d}$ to the graph of $f$ is an isomorphism. We set $\alpha\coloneqq\pr_{1}|_{(\Gamma_{f})_{y}}$. By pulling back $\alpha$ along $\bbG^{\varphi(d)}_{m}\hookrightarrow{(\sP_{d})}_{y}$, we get an isomorphism $\alpha\colon\alpha^{-1}(\bbG^{\varphi(d)}_{m})\longrightarrow\bbG^{\varphi(d)}_{m}$. We claim that $\alpha^{-1}$ is the graph of a homomorphism $\bbG^{\varphi(d)}_{m}\longrightarrow\bbG^{\varphi(d)}_{m}$. Indeed, it suffices to show that $\pr_{2}(\alpha^{-1}(\bbG^{\varphi(d)}_{m}))\subset\bbG^{\varphi(d)}_{m}$. To see this, observe that the composite \[\bbG^{\varphi(d)}_{m}\overset{\cong}\longrightarrow \alpha^{-1}(\bbG^{\varphi(d)}_{m})\subset (\Gamma _{f})_{y}\overset{\pr_{2}}{\longrightarrow}(\sP_{d})_{y}\longrightarrow \sP_{d}(\tilde{C'_{y}}/\tilde{C_{y}})\] is the zero map by \cite[Cor.\ 3.9]{mil1} and hence, $\pr_{2}|_{\bbG^{\varphi(d)}_{m}}$ factors through the kernel of $(\sP_{d})_{y}\longrightarrow \sP_{d}(\tilde{C'_{y}}/\tilde{C_{y}})$ which is $\bbG^{\varphi(d)}_{m}$. Finally, we define $\rho(f)$ to be this endomorphism of $\bbG^{\varphi(d)}_{m}$. One checks that $\rho$ is a homomorphism of rings.
\end{proof}
\end{step}
\begin{proof}[Conclusion]Eventually, we are in the position to complete the proof. Suppose $\End_{\mu_{d}}((\sP_{d})_{\bar{\eta}})\neq \bbZ[\zeta_{d}]$ and choose a $\mu_{d}$-equivariant endomorphism $f$ not in $\bbZ[\zeta_{d}]$. The endomorphism $f$ can be described as a $\kappa(\bar{\eta})$-point of $\End^{\mu_{d}}_{\sP_{d}/U}$ and we let $Z\subset \End^{\mu_{d}}_{\sP_{d}/U}$ be the irreducible component containing this point.  Then, the generic point $\theta\in Z$ corresponds to a $\mu_{d}$-equivariant endomorphism not in $\bbZ[\zeta_{d}]$. Consider the finite set \[\Gamma\coloneqq \{ n\coloneqq(n_{0},n_{1},\ldots,n_{\varphi(d)-1})\in \bbZ^{\varphi(d)}\ | \ \im ([n]\footnote{$[n]\coloneqq n_{0}\id+n_{1}\sigma^{*}+n_{2}(\sigma^{*})^{2}+\dots+n_{\varphi(d)-1}(\sigma^{*})^{\varphi(d)-1}$})\cap Z\neq \emptyset \}.\] Each $\im ( [n]  )\cap Z$ is a proper closed subset of $Z$. Setting \[Z_{n}\coloneqq\pi(\im( [n]  )\cap Z),\] for $n\in \Gamma$, we get finitely many nowhere dense closed subsets of $U$, such that for every point $u\in U\setminus \bigcup_{n\in\Gamma }Z_{n}$ the fibre $\pi^{-1}(u)$ contains a point, which is not in $\bbZ[\zeta_{d}]$. We can choose a Lefschetz pencil as above, such that $(\sP_{d})_{\bar{\mu}}$ is $\mu_{d}$-simple, cf. Step \ref{2} and $\End_{\mu_{d}}((\sP_{d})_{\bar{\mu}})\neq \bbZ[\zeta_{d}]$. By Step \ref{3} this leads to a contradiction.
\end{proof}
\end{proof}
The next lemma consists of the final reduction step.
\begin{lemma}\label{lem:3.4} The abelian variety $(\sP_{d})_{\eta}$ is $\mu_{d}$-simple if and only if it is $\mu_{d}$-absolutely simple.
\end{lemma}
\begin{proof} Clearly, if $(\sP_{d})_{\eta}$ is $\mu_{d}$-absolutely simple, then it is $\mu_{d}$-simple. Conversely, assume that $(\sP_{d})_{\eta}$ is $\mu_{d}$-simple but not $\mu_{d}$-absolutely simple. Then, there is a finite field extension $L\supset\kappa(\eta)$ and a non-zero and proper $\mu_{d}$-simple abelian subvariety $B$ of $(\sP_{d})_{L}$, such that $(\sP_{d})_{L}$ can be written up to isogeny as a product $\prod B^{\tau}$, where $B^{\tau}$ stands for a Galois conjugate of $B$ and $\tau$ runs through a finite subset $J\subset\Gal(L/\kappa(\eta))$ of cardinality greater equal to 2. The field extension $L\supset\kappa(\eta)$ gives rise to a morphism $g\colon U'\longrightarrow U$, which we may assume is étale. For $\tau\in J$, we let $\varphi_{\tau}$ be the endomorphism of $(\sP_{d})_{L}$ whose image is $B^{\tau}$. More explicitly, $\varphi_{\tau}$ is given by \[(\sP_{d})_{L}\overset{\sim}\longrightarrow \prod B^{\tau}\overset{proj}\longrightarrow B^{\tau}\subset(\sP_{d})_{L}.\] Pick a Lefschetz pencil $(C_{t})_{t\in\bbP^{1}}$, such that its singular members are irreducible and intersect the branch locus $\Delta$ transversally. Let $X$ be any irreducible component of $g^{-1}(\bbP^{1}\cap U)$. Then, $X$ dominates $\bbP^{1}\cap U$ and if $\theta\in X$ is its generic point, then each $\varphi_{\tau}$ determines an endomorphism of $\sP_{d}$ over $\theta$, e.g. using the N\'eron mapping property, such that if $B^{\tau}\coloneqq\im(\varphi_{\tau})$, then $\prod B^{\tau}\sim (\sP_{d})_{\theta}$. Let $\bar{X}$ be a smooth compactification of $X$ and $\bar{X}\longrightarrow \bbP^{1}$ the extension of $g\colon X\longrightarrow \bbP^{1}\cap U$. Fix a point $y\in \bar{X}$ lying over a point of the pencil which corresponds to a nodal curve and consider the local ring $R$ of $\bar{X}$ at $y$. Since $\sP_{d}$ admits a semi-abelian reduction over $R$ , cf. \eqref{lem:3.2} the same is true for all $B^{\tau}$, cf. \cite[Cor.\ 7.1.6]{blr}. We denote by $\tilde{B}^{\tau}$ the identity component of the N\'eron model of $B^{\tau}$. Then, the isogeny of the generic fibre extends to an isogeny $\prod \tilde{B}^{\tau}\sim\sP_{d}$ over $R$, cf. \cite[Prop.\ 7.3.6]{blr}. Since $(\sP_{d})_{y}$ is an extension of an abelian variety by a torus of rank $\varphi(d)$, cf. \eqref{lem:3.2}, it follows that the toric part of $\tilde{B}^{\tau}_{y}$ has rank $\delta,\ 1\leq\delta\leq\varphi(d)$, such that $\delta|J|=\varphi(d)$. As in Step \ref{3}, one constructs a homomorphism $\rho_{\tau} \colon \End_{\mu_{d}}(B^{\tau}) \longrightarrow \End(\mathbb{G}^{\delta}_{m})$. Since the restriction of $\psi\circ\rho_{\tau}$ to $\bbZ[\zeta_{d}]\subset\End_{\mu_{d}}(B^{\tau})$ is injective, where $\psi\coloneqq\pr_{1}\circ-\colon\End(\bbG^{\delta}_{m})\longrightarrow\Hom(\bbG^{\delta}_{m},\bbG_{m})$ we find that $\delta=\varphi(d)$ and $|J|=1$, which is absurd.\end{proof}
\end{section}
\begin{section}{The Proof of Theorem \ref{thm:1.1}}
According to the results of Section \ref{sec:3}, our task to prove Theorem \ref{thm:1.1} is reduced to showing $(\sP_{d})_{\eta}$ is a $\mu_{d}$-simple abelian variety. Recall, that we have an isogeny \[\Jac(C'_{\eta})\sim \Jac(C_{\eta})\times \prod_{d|n,\ d\neq 1} (\sP_{d})_{\eta}.\] Given a non-zero endomorphism $\varepsilon\in\End_{\mu_{d}}((\sP_{d})_{\eta})$. Then, by considering the composite \[{ \varepsilon }'\colon \Jac(C'_{\eta})\overset{\sim }{\longrightarrow} \Jac(C_{\eta})\times \prod_{d|n,\ d\neq 1}(\sP_{d})_{\eta}\overset{\pr_{d}}{\longrightarrow}(\sP_{d})_{\eta}\overset{\varepsilon}\longrightarrow(\sP_{d})_{\eta}\hookrightarrow \Jac(C'_{\eta}),\]we get a $\mu_{d}$-equivariant endomorphism of $\Jac(C'_{\eta})$ whose restriction to $(\sP_{d})_{\eta}$ is simply $\varepsilon\circ[n]$. Hence, it suffices to show that that the restriction of $\varepsilon'$ to $(\sP_{d})_{\eta}$ lies in $\bbZ[\zeta_{d}]$. Recall, that abelian schemes satisfy a stronger N\'eron mapping property, cf. \cite[Sec.\ 3.1.5]{ro}. Thus, the endomorphism $\varepsilon'$ extends to an endomorphism \[\varepsilon'\colon\Pic^{0}_{\sY/U}\longrightarrow \sP_{d}\subset\Pic^{0}_{\sY/U}.\] Let $[T]\in\Corr_{U}(\sY)$ be the class of a correspondence $T$ on $\sY\times_{U}\sY$ associated to the endomorphism $\varepsilon'$, cf. \eqref{prop:2.2}. We write $T=\sum n_{i}T_{i}$, where $T_{i}$ are prime divisors. Let $\Sigma$ be a general two dimensional linear system in $|\sH|_{x}$, i.e. the general member of $\Sigma$ is smooth and intersects the branch locus $\Delta$ transversally. Then, the correspondences $T_{i}$ are all defined over a non-empty open subset of $\Sigma$ and we can construct a rational map \begin{tikzcd}
{\phi _{\Sigma,T_{i}}\colon{ S}'} \arrow[r, dashed] & \Div^{+}({ S}'),\ y\mapsto \Gamma^{i}_{y}
\end{tikzcd}
, cf. \cite[pp.\ 38]{cv}. Especially, we get a rational map\begin{center} \begin{tikzcd}
{\phi _{\Sigma,T}\colon{ S}'} \arrow[r, dashed] & \Pic({ S}'),\ y\mapsto [\Gamma_{y}]\coloneqq\sum n_{i}[\Gamma^{i}_{y}].
\end{tikzcd}\end{center}
\par Let $[ C ]\in | \sH  |_{x}$ be a general member and choose a general two-dimensional linear system $\Sigma$ containing $[C]$. Consider the rational map $\phi _{\Sigma,T}$. Then, for a general point $y\in {C}'$ we get a divisor $\Gamma _{y}=\phi _{\Sigma,T}(y)$ on ${S}'$. Set $w=f(y)\in C$, $f^{-1}(w)= \{ y,\sigma(y),\ldots, \sigma^{n-1}(y)\}$ and $f^{-1}(x)= \{ z,\sigma(z),\ldots,\sigma^{n-1}(z)\}$, where $\sigma$ is a generator of the Galois group of the covering $f$. The following lemma computes the divisor $E_{y}$ in ${C}'$ corresponding to the intersection of ${C}'$ with $\Gamma _{y}$.
\begin{lemma}\label{lem:4.1} We have that $E_{y}={\alpha}_{0} z+{\alpha}_{1}\sigma(z)+\ldots+{\alpha}_{n-1}\sigma^{n-1}(z)+{\beta}_{0} y+ {\beta}_{1}\sigma(y)+\ldots+{\beta}_{n-1}\sigma^{n-1}(y)+\gamma {\sB}'_{x,w}+T_{{C}'}(y)$, where $\alpha_{i} ,\beta_{i} ,\gamma \in\mathbb{Z}$ and ${\sB}'_{x,w}$ is the pull-back of the divisor of base points different from $x$ and $w$ of $\Sigma_{w}$ under the covering $f$.
\end{lemma}
\begin{proof} Cf. \cite[Lem.\ 3.6]{cv}.
\end{proof}
\subsection{Regular case}\label{10}The branched locus $\Delta$ of the covering $f$ is a smooth ample divisor and thus, the canonical map $\Alb(f)\colon\Alb(S')\longrightarrow\Alb(S)$ induced by $f$ is an isomorphism. Indeed, since $f_{*}\sO_{S'}\cong\bigoplus^{n-1}_{i=0}\sL^{-i}$, the Kodaira Vanishing theorem gives $H^{1}(\sO_{S'})=H^{1}(\sO_{S})$ and hence, $\Alb(f)$ is an isogeny. From this one immediately sees that the induced action on $\Alb(S')$ is trivial, i.e. $\Alb(\sigma)=\id$. Consider the Albanese map $\Alb_{\xi_{o}}\colon S'\longrightarrow \Alb(S')$, where the point $\xi_{o}\in S'$ lies over a point of the branch locus $\Delta\subset S$ and observe that the map is invariant under the $\mu_{n}$-action. Therefore, we find a homomorphism $\Alb(S)\longrightarrow\Alb(S')$ that is inverse to $\Alb(f)$, proving the claim. In particular, we deduce that $q(S)=q(S')$. Here, we give the proof for the case $S$ is regular, i.e. $q(S)=0$.
\begin{proof}[Proof of Theorem \ref{thm:1.1} for the regular case]If $S$ is regular, then $\Pic({S}')$ is discrete and thus, the rational map $\phi _{\Sigma,T}$ is constant. Hence, for a general point $y\in {C}'$, the curves $\Gamma _{y}$ and $\Gamma _{\sigma(y)}$ are linearly equivalent. It follows that $E_{y}$ and $E_{\sigma(y)}$ are also linearly equivalent and so, $E _{y}-E_{\sigma(y)}={\beta}_{0} (y-\sigma(y))+ {\beta}_{1}\sigma(y-\sigma(y))+\ldots+{\beta}_{n-1}{\sigma}^{n-1}(y-\sigma(y))+T_{{C}'}(y-\sigma(y))\sim 0$. Since $\Prym(C'/C)=\im(\id-\sigma^{*})$, the latter forces $T_{{C}'}(y)=(-{\beta}_{0})y+\ldots+(-{\beta}_{n-1})\sigma^{n-1}(y)$ for all $y\in\Prym(C'/C)$. Eventually, we see that the restriction of $T_{C'}$ to $\sP_{d}(C'/C)$ takes the desired form. This yields that the restriction of ${\varepsilon}'$ to $(\sP_{d})_\eta$ lies in $\mathbb{Z}[\zeta_{d}]$, as claimed.\end{proof}
\subsection{Irregular case} We shall use the following lemma.
\begin{lemma} \label{lem:4.2} Let $a\colon \Jac({C}')\longrightarrow \sP_{d}(C'/C)\subset \Jac({C}')$ be a homomorphism and let $T_a$ be a correspondence associated to it, cf. \eqref{prop:2.2}. Assume that there exist $\alpha_{0},\ldots,\alpha_{n-1} \in \mathbb{Z}$, such that for general $y\in{C}'$ the divisor class $T_{a}(y-\sigma(y))+ {\alpha}_{0}(y-\sigma(y))+\ldots+{\alpha}_{n-1}{\sigma}^{n-1}(y-\sigma(y))$ lies in $\Pic^{0}({S}')$. Then, the restriction of $a$ to $\sP_{d}(C'/C)$ lies in $\mathbb{Z}[\zeta_{d}]\subset\End( \sP_{d}(C'/C))$.
 \end{lemma}  
\begin{proof} As $\Prym(C'/C)=\im(\id-\sigma^{*})$, the assumption clearly implies that $a(y)+ {\alpha}_{0}y+\ldots+{\alpha}_{n-1}{\sigma}^{n-1}(y) \in \im(i^{*})$ for all $y\in\Prym(C'/C)$, where $i^{*}\colon\Pic^{0}({S}')\longrightarrow \Pic^{0}({C}')=\Jac({C}')$ is the natural pull-back induced by ${C}'\hookrightarrow {S}'$. If we show that the intersection $\im(i^{*})\cap \Prym(C'/C)$ is finite, then the result follows immediately. Indeed, consider the commutative square: 
 \begin{center}
 \begin{tikzcd}
\Pic^{0}(S) \arrow[d, "f^{*}"] \arrow[r, "i^{*}"] & \Pic^{0}(C) \arrow[d, "f^{*}"] \\
\Pic^{0}({S}') \arrow[r, "i^{*}"]                 & \Pic^{0}({C}')     .           
\end{tikzcd}
\end{center}
The canonical map $\Alb({S}')\longrightarrow\Alb(S)$ induced by $f$ is an isomorphism and so, is its dual, which is $f^{*}$. Hence, the latter yields that $\im(i^{*}\colon \Pic^{0}({S}')\longrightarrow \Pic^{0}({C}'))\subset f^{*}(\Pic^{0}(C))$. By the definition of $\Prym(C'/C)$, we know that $f^{*}(\Pic^{0}(C))\cap\Prym(C'/C)$ is finite and so, is the intersection $\im(i^{*})\cap\Prym(C'/C)$, proving the claim.\end{proof}
 \begin{proof}[Proof of Theorem \ref{thm:1.1} for the irregular case] Using the curves $\Gamma _{y}$ we find that $E _{y}-E_{\sigma(y)}$ lies in the image of $\Pic({S}')\longrightarrow \Pic({C}')$. Therefore, we have that $T_{{C}'}(y-\sigma(y))+{\beta}_{0} (y-\sigma(y))+ {\beta}_{1}\sigma(y-\sigma(y))+\ldots+{\beta}_{n-1}{\sigma}^{n-1}(y-\sigma(y)) \in \im(i^{*}\colon\Pic({S}')\longrightarrow \Pic({C}'))$ for general $y\in C'$. It follows that ${\varepsilon}'\in \mathbb{Z}[\zeta_{d}]\subset \End((\sP_{d})_{\eta})$, cf. \eqref{lem:4.2}.\end{proof}
\end{section}
\begin{section}{The proof of Theorem \ref{thm:1.3}} The proof is similar to the case of \eqref{thm:1.1}. First, we need to replace our earlier family $\varphi_{d}\colon\sP_{d}\longrightarrow U$. In particular, we consider the abelian fibration \[\sR_{d}\coloneqq\ker^{0}(\sP_{d}\longrightarrow \Alb(S')\times U).\] Assume that the abelian fibration $\varphi_{d}\colon\sR_{d}\longrightarrow U$ is non-zero, i.e. $\sR_{[C]}\neq0$ for $[C]\in U$. Then, we show that for the very general member $[C]\in U$, we have that $\End_{\mu_{d}}((\sR_{d})_{[C]})\cong\bbZ[\zeta_{d}]$. One checks that the results $\eqref{prop:3.3}$ and $\eqref{lem:3.4}$ still hold true for the family $\varphi_{d}\colon\sR_{d}\longrightarrow U$.\par We proceed as in the proof of Theorem \ref{thm:1.1}. A non-zero endomorphism $\varepsilon\in\End_{\mu_{d}}((\sR_{d})_{\eta})$ gives rise to a $\mu_{d}$-equivariant endomorphism $\varepsilon'\in\End(\Jac(C'_{\eta}))$ and it is enough to check that the restriction of $\varepsilon'$ to $(\sR_{d})_{\eta}$ lies in $\bbZ[\zeta_{d}]$. The following lemma is needed.
\begin{lemma}\label{lem:5.1} Let $a\colon \Jac({C}')\longrightarrow \sR_{d}(C',C,{S}')\subset \Jac({C}')$ be a homomorphism and let $T_a$ be a correspondence associated to it, cf. \eqref{prop:2.2}. Assume that there exist $\alpha_{0},\ldots,\alpha_{n-1} \in \mathbb{Z}$, such that for general $y\in{C}'$ the divisor class $T_{a}(y-\sigma(y))+ {\alpha}_{0}(y-\sigma(y))+\ldots+{\alpha}_{n-1}{\sigma}^{n-1}(y-\sigma(y))$ lies in $\Pic^{0}({S}')$. Then, the restriction of $a$ to $\sR_{d}(C',C,{S}')$ lies in $\bbZ[\zeta_{d}]\subset\End( \sR_{d}(C',C,{S}'))$.\end{lemma}  
\begin{proof} Clearly, we have that $a(y)+{\alpha}_{0}y+\ldots+{\alpha}_{n-1}{\sigma}^{n-1}(y)\in\im(i^{*})$ for all $y\in\Prym(C'/C)$, where $i^{*}\colon\Pic^{0}({S}')\longrightarrow \Pic^{0}({C}')=\Jac({C}')$ is the pull-back induced by ${C}'\hookrightarrow {S}'$. Let $\sK(C',S')\coloneqq\ker(\Jac(C')\longrightarrow\Alb(S'))$ and observe that the intersection $\im(i^{*})\cap \sK(C',S')$ is finite. Since $\sR_{d}(C',C,{S}')\subset \sK(C',S')$, we find that $a(y)+{\alpha}_{0}y+\ldots+{\alpha}_{n-1}{\sigma}^{n-1}(y)=0$ for all $y\in \sR_{d}(C',C,{S}')$. Therefore, the restriction of $a$ to $\sR_{d}(C',C,{S}')$ belongs to $\bbZ[\zeta_{d}]$, as claimed.\end{proof}
\begin{proof}[Proof of Theorem \ref{thm:1.3}] Using the curves $\Gamma _{y}$ one sees that $E _{y}-E_{\sigma(y)}$ lies in the image of $\Pic({S}')\longrightarrow \Pic({C}')$. It follows that $T_{{C}'}(y-\sigma(y))+{\beta}_{0} (y-\sigma(y))+ {\beta}_{1}\sigma(y-\sigma(y))+\ldots+{\beta}_{n-1}{\sigma}^{n-1}(y-\sigma(y)) \in \im(i^{*}\colon\Pic({S}')\longrightarrow \Pic({C}'))$. Now, the result is an immediate consequence of \eqref{lem:5.1}.\end{proof}
\end{section}
\printbibliography
\vspace{1cm}
\begin{flushleft}
\address{MATHEMATISCHES INSTITUT, UNIVERSIT\"AT BONN, ENDENICHER ALLEE 60, 53115 BONN, GERMANY}\\
\address{{\bf{Current address}}: INSTITUT F\"UR ALGEBRAISCHE GEOMETRIE, LEIBNIZ UNIVERSIT\"AT HANNOVER, WELFENGARTEN 1, 30167 HANNOVER, GERMANY.}\\
\email{{\bf{Email address:}\ }alexandrou@math.uni-hannover.de}
\end{flushleft}
\end{document}